\documentclass{article}

\usepackage[english]{babel}

\usepackage{amsmath,amssymb,setspace}
\usepackage{amsthm}
\usepackage[english]{babel}
\usepackage{amsfonts}
\usepackage{calligra}
\usepackage[T1]{fontenc}
\usepackage{xstring}
\usepackage{graphics}
\usepackage[numbers,sort&compress]{natbib}
\usepackage{multicol}
\date{}
\usepackage{longtable}
\usepackage{float}
\usepackage[margin=0.7 in]{geometry}
\usepackage{graphicx}
\usepackage{mathtools}
\usepackage{amsmath}
\usepackage{subcaption}
\usepackage{amsmath, amssymb, setspace}

\usepackage{amsthm}
\usepackage[english]{babel}
\usepackage{amsfonts}
\usepackage{calligra}
\usepackage[T1]{fontenc}
\usepackage{xstring}

\usepackage{amssymb}
\usepackage{amsthm}
\usepackage{amsmath}
\usepackage{lipsum}
\newtheorem{definition}{Definition}
\theoremstyle{plain}
\theoremstyle{definition}
\theoremstyle{remark}
\newtheorem{theorem}{Theorem}

\newtheorem{lemma}{Lemma}
\newtheorem{remark}{Remark}
\newtheorem{example}{Example}

\title{Second $\alpha$-Order Fractal Differential Equations}
\author{Alireza Khalili Golmankhaneh$^1$, Donatella Bongiorno$^2$,\\
$^1$ Department of Physics, Urmia Branch, Islamic Azad University, Urmia 63896,West Azerbaijan,  Iran\\
alirezakhalili2002@yahoo.co.in\\
$^2$ Department of Engineering, University of Palermo,
Palermo 90100, Italy\\
donatella.bongiorno@unipa.it\\
}

\begin{document}

\maketitle

\let\thefootnote\relax
\footnote{Corresponding author Alireza Khalili Golmankhaneh}
\begin{abstract}
In this research paper, we provide a concise overview of fractal calculus applied to fractal sets. We introduce and solve a second $\alpha$-order fractal differential equation with constant coefficients across different scenarios. We propose a uniqueness theorem for second $\alpha$-order fractal linear differential equations. We define the solution space as a vector space with non-integer orders. We establish precise conditions for second $\alpha$-order fractal linear differential equations and derive the corresponding fractal adjoint differential equation.
\end{abstract} 

\section{Introduction}
Fractal patterns, albeit in an array of scales rather than in an infinite manner, have been modeled extensively due to the time and space-related limits concerning practice-wise elements. It is possible that the models might simulate theoretical fractals or natural phenomena with fractal features, and the results derived from modeling processes can be employed as benchmarks for fractal analysis purposes. Fractal calculus, which emerged as a formulation extending ordinary calculus, procures a constructive and algorithmic approach toward the smooth differentiable-structured modeling of natural processes through fractals.\\
Benoit Mandelbrot is credited with pioneering the field of fractal geometry \cite{Mandelbro}, which revolves around shapes possessing fractal dimensions that surpass their topological dimensions \cite{falconer1999techniques,jorgensen2006analysis}. These intricate fractals exhibit self-similarity and frequently demonstrate non-integer and complex dimensions \cite{Qaswet,Lapidus}. Nevertheless, the analysis of fractals presents challenges, given that traditional geometric measures such as Hausdorff measure \cite{rogers1998hausdorff}, length, surface area, and volume are typically applied to standard shapes \cite{Ewqq}. Consequently, the direct application of these measures to fractal analysis becomes intricate \cite{Barnsley,Gregory,rosenberg2020fractal,Tosatti,bishop2017fractals,Shlomo}.

Researchers have addressed the challenge of fractal analysis by employing various methodologies. Among these approaches, harmonic analysis \cite{kigami2001analysis,Strichartz2}, measure theory \cite{giona1995fractal,freiberg2002harmonic,jiang1998some,Bongiorno23,bongiorno2018derivatives,bongiorno2015fundamental,bongiorno2015integral}, fractional space and nonstandard methods \cite{stillinger1977axiomatic}, probabilistic methods \cite{Barlow} fractional calculus \cite{e25071008,uchaikin2013fractional,damian2023mechanical,samayoa2022map,Trifcebook} and non standard methods \cite{nottale2011scale}.

Fractal calculus is a mathematical framework that extends traditional calculus, allowing for the treatment of equations whose solutions take the form of functions exhibiting fractal properties, such as fractal sets and curves \cite{parvate2009calculus, parvate2011calculus}. What makes fractal calculus particularly appealing lies in its elegance and algorithmic methodologies, which contrast favorably with other techniques \cite{Alireza-book}.

The generalization of $F^{\alpha}$-calculus (FC) has been successfully achieved by employing the gauge integral method. This generalization focuses on the integration of functions over a specific subset of the real number line that contains singularities found within fractal sets \cite{golmankhaneh2016fractal,golmankhaneh2023fuzzification}.

Various methods have consequently been employed to solve fractal differential equations, and their stability conditions have been determined accordingly \cite{golmankhaneh2019sumudu,Fourier1,golmankhaneh2023solving}.

The application of FC is demonstrated through the analysis of fractal interpolation functions and Weierstrass functions. These functions often display characteristics of non-differentiability and non-integrability when viewed through the lenses of traditional calculus \cite{gowrisankar2021fractal}.

Fractal calculus has been expanded to encompass the study of Cantor cubes and Cantor tartan \cite{golmankhaneh2018fractalt}, and within this framework, the Laplace equation has been formally defined \cite{khalili2021laplace}.

Analyses based on nonlocal fractal calculus have been conducted for electrical circuits with arbitrary source terms, extending the examination to circuits \cite{banchuin20224noise}.

Numerical simulations explored the influence of parameters on noise performance. Increasing the orders of fractal-fractional reactive components generally improved noise performance across different circuits \cite{Rewid3,banchuin2022noise}.

The utilization of non-local fractal derivatives to characterize fractional Brownian motion on thin Cantor-like sets was demonstrated. The proposal of the fractal Hurst exponent establishes its connection to the order of non-local fractal derivatives \cite{Alireza-book}.
Furthermore, fractal stochastic differential equations have been defined, with categorizations for processes like fractional Brownian motion and diffusion occurring within mediums with fractal structures \cite{Alireza-book,khalili2019fractalcat,khalili2019random}.
Local vector calculus within fractional-dimensional spaces, on fractals, and in fractal continua was developed. The proposition was put forth that within spaces characterized by non-integer dimensions, it was feasible to define two distinct del-operators-each operating on scalar and vector fields. Employing these del-operators, the foundational vector differential operators and Laplacian in fractional-dimensional space were formulated conventionally. Additionally, Laplacian and vector differential operators linked with $F^{\alpha}$-derivatives on fractals were established \cite{balankin2023vector}. The concept of a fractal comb and its associated staircase function was introduced. Derivatives and integrals were defined for functions on these combs using the staircase function \cite{golmankhaneh2023fractal}.
Fractal retarded, neutral, and renewal delay differential equations with constant coefficients were solved through the utilization of the method of steps and the Laplace transform \cite{golmankhaneh2023initial}.
Fractal integral and differential forms were defined through nonstandard analysis \cite{khalili2023non}.
Fractal time, recently suggested by physics researchers due to its self-similar properties and fractional dimension, was investigated in the context of economic models employing both local and non-local fractal Caputo derivatives
\cite{faghih2023introduction,Alireza-book,Welch-5,Vrobel,Shlesinger-6,plonka1995fractal}

Along these lines of developments and implications, we have introduced second $\alpha$-order fractal linear differential equations and elucidated their corresponding solutions.\\
To this end, the current paper is structured as follows: in Section \ref{1g}, we provide a concise review of fractal calculus.
Moving on to Section \ref{2g}, we define and solve the second $\alpha$-order fractal differential equation and present the uniqueness theorem associated with it.
In Section \ref{3g}, we extend the discussion to encompass the uniqueness theorem for second $\alpha$-order fractal linear differential equations.
Section \ref{4g} addresses the comprehensive study of the exact second $\alpha$-order fractal differential equation.
In Section \ref{5g}, we delve into the solution and presentation of the nonhomogeneous second $\alpha$-order fractal differential equation.
Finally, Section \ref{6g} provides the conclusion, discussion, and future.

\section{Overview of Fractal Calculus on Fractal Sets \label{1g}}
In this section, we present a concise overview of fractal calculus applied to fractal sets as summarized in \cite{parvate2009calculus,parvate2011calculus,Alireza-book},  moreover in this section and more generally throughout the paper, we will indicate by $F$ an $\alpha$-perfect fractal subset of real line.

\begin{definition}
The flag function of a set $F$ and a closed interval $I\subset (a,b)$ is defined as:
\begin{equation}
  \rho(F,I)=
  \begin{cases}
    1, & \text{if } F\cap I\neq\emptyset;\\
    0, & \text{otherwise}.
  \end{cases}
\end{equation}
\end{definition}

\begin{definition}
For  a subdivision $P_{[a,b]}$ of $[a,b]$, and for a given $\delta>0$, the coarse-grained mass of $F\cap [a,b]$ is defined by
\begin{equation}
  \gamma_{\delta}^{\alpha}(F,a,b)=\inf_{|P|\leq
\delta}\sum_{i=0}^{n-1}\Gamma(\alpha+1)(t_{i+1}-t_{i})^{\alpha}
\rho(F,[t_{i},t_{i+1}]),
\end{equation}
where $|P|=\max_{0\leq i\leq n-1}(t_{i+1}-t_{i})$, and $0< \alpha\leq1$.
\end{definition}

\begin{definition}
The mass function of  $F$ is defined as the limit of the coarse-grained mass as $\delta$ approaches zero:
\begin{equation}
  \gamma^{\alpha}(F,a,b)=\lim_{\delta\rightarrow0}\gamma_{\delta}^{\alpha}(F,a,b).
\end{equation}
\end{definition}

\begin{definition}
The $\gamma$-dimension of $F\cap [a,b]$ is defined as:
\begin{align}
  \dim_{\gamma}(F\cap
[a,b])&=\inf\{\alpha:\gamma^{\alpha}(F,a,b)=0\}\nonumber\\&
=\sup\{\alpha:\gamma^{\alpha}(F,a,b)=\infty\}
\end{align}
\end{definition}

\begin{definition}
The integral staircase function of order $\alpha$ for  $F$ is given by:
\begin{equation}
 S_{F}^{\alpha}(x)=
 \begin{cases}
   \gamma^{\alpha}(F,a_{0},x), & \text{if } x\geq a_{0}; \\
   - \gamma^{\alpha}(F,x,a_{0}), & \text{otherwise}.
 \end{cases}
\end{equation}
where $a_{0}$ is an arbitrary fixed real number.
\end{definition}

\begin{definition}
Let  $f$ be a real function defined on $F \cap(a,b).$ Let $x$ be a point in $F\cap [a,,b].$ Whenever $$\underset{ y\rightarrow x}{F-\text{lim}}~f(y)\,=\,f(x),$$ we say that the function $f$ is $F$-continuous at $x$.
If $f$ is $F$-continuous at each points of $F \cap(a,b)$ we say that $f$ is $F$-continuous in $(a,b).$
\end{definition}

\begin{remark}
Let us observe that the continuity of a function $f$ in the interval $(a,b)$ implies the $F$-continuity of $f$ in $(a,b),$ however the converse is not true (see \cite{parvate2009calculus,parvate2011calculus,Alireza-book}).
\end{remark}

\begin{definition}
Let $f$ be a real function defined in $(a,b).$ The $F^{\alpha}$-derivative of $f$  is defined as follows:
\begin{equation}
  D_{F,x}^{\alpha}f(x)=
  \begin{cases}
    \underset{ y\rightarrow
x}{F{-}\text{lim}}~\frac{f(y)-f(x)}{S_{F}^{\alpha}(y)-S_{F}^{\alpha}(x)}, & \text{if } x\in F; \\
    0, & \text{otherwise}.
  \end{cases}
\end{equation}
if  $F_{-}\text{lim}$ exists \cite{parvate2009calculus}.  Moreover, if $f$ is $F^{\alpha}$-derivable at each points of the fractal set $F,$ we say that $f$ is $F^{\alpha}$-derivable in $(a,b).$

\end{definition}

\begin{remark}
Let us observe that the derivative of a function $f$ in the interval $(a,b)$ implies the $F$-derivative of $f$ in $(a,b),$ however the converse is not true.
Indeed it is trivial to observe that the $F^{\alpha}$-derivative of the characteristic function of the fractal set $F$ is not derivable in the usual sense at a point of $(a,b)$ since it is not continuous.
\end{remark}

\begin{theorem}
Let $f$ be a real function defined in $(a,b).$  If $f$ is $F^{\alpha}$-derivable in $(a,b),$ then the function $f$ is $F$-continuous in $(a,b).$
\end{theorem}

\begin{definition}
Let  $f$ be a real function defined in $(a,b)$ and let $x$ be a point in $F \cap(a,b)$.  If there exists the $F^{\alpha}$-derivative of $f$ at the point $x,$ the second $F^{\alpha}$-derivative of $f$ at the same point $x$ is defined as follows:
\begin{equation}
  (D_{F,x}^{\alpha})^2f(x)=
  \begin{cases}
    \underset{ y\rightarrow
x}{F{-}\text{lim}}~\frac{D_{F,y}^{\alpha}f(y)-D_{F,x}^{\alpha}f(x)}{S_{F}^{\alpha}(y)-S_{F}^{\alpha}(x)}, & \text{if } x\in F; \\
    0, & \text{otherwise}.
  \end{cases}
\end{equation}
if $F{-}\text{lim}$ exists.
\end{definition}

\begin{remark}
The classical second derivative of a function $f$  at a point $x\in F \cap(a,b)$  implies the second $F^{\alpha}$-derivative of $f$ at the same point $x.$
 However the converse it is not true.  Indeed it is enough to consider the integral staircase function of order $\alpha$ for  $F$, since $D_{F,x}^{\alpha}S^{\alpha}_F(x)=\chi_F(x),$ and
 $\chi_F(x),$ is not derivable in the usual sense at a point $x\in F \cap(a,b)$  since it is not continuous at the same point.
 \end{remark}

\begin{definition}
Let $n\geq 2$, and let $f$ be a real function defined in $(a,b).$ Let $x$ be a point in $F \cap(a,b)$. If there exists the $(n-1)^{th}\,F^{\alpha}$-derivative of $f$ at the point $x,$ the $n^{th}\,F^{\alpha}$-derivative of $f$ at the same point $x$ is defined as follows:
\begin{equation}
 (D_{F,x}^{\alpha})^nf(x)=
  \begin{cases}
    \underset{ y\rightarrow
x}{F{-}\text{lim}}~\frac{(D_{F,y}^{\alpha})^{(n-1)}f(y)-(D_{F,x}^{\alpha})^{(n-1)}f(x)}{S_{F}^{\alpha}(y)-S_{F}^{\alpha}(x)}, & \text{if } x\in F; \\
    0, & \text{otherwise}.
  \end{cases}
\end{equation}
if  $F{-}\text{lim}$ exists.

A function $f$ that has the $n^{th}\,F^{\alpha}$-derivative at the point $x$  is said to be $n$-times $F^{\alpha}$ differentiable at $x$.
\end{definition}

\begin{remark}
 The  $n^{th}\,F^{\alpha}$-derivative of $f$ at the point $x$ will be denoted also by: $D_{F,x}^{n\alpha}f(x).$
Therefore, from now on:
$$(D_{F,x}^{\alpha})^nf(x)=D_{F,x}^{n\alpha}f(x),~~~\forall x\in F.$$
\end{remark}

\begin{definition}
If $f$ is $n$-times $F^{\alpha}$ differentiable at each point $x$ of the fractal set $F \cap(a,b)$ we say that $f$ is $n$-times $F^{\alpha}$- differentiable in $(a,b)$.
The collection of all functions $f\colon (a,b)\to \mathbb{R}$ that are
$n$-times $F^{\alpha}$-differentiable in $(a,b)$ will be denoted by $C^{n,\alpha}(a,b)$
\end{definition}

\begin{definition}
Let $I=[a,b]$ such that $S^{\alpha}_F$ is finite on $I$. Let $f$ be a bounded function defined on $(a,b)$ and let $x\in F \cap I$. The $F^{\alpha}$-integral of $f$ on $I$ is defined as:
\begin{align}
  \int_{a}^{b}f(x)d_{F}^{\alpha}x&=\sup_{P_{[a,b]}}
\sum_{i=0}^{n-1}\inf_{x\in F\cap
I}f(x)(S_{F}^{\alpha}(x_{i+1})-S_{F}^{\alpha}(x_{i}))
\nonumber\\&=\inf_{P_{[a,b]}}
\sum_{i=0}^{n-1}\sup_{x\in F\cap
I}f(x)(S_{F}^{\alpha}(x_{i+1})-S_{F}^{\alpha}(x_{i})).
\end{align}

\end{definition}

\section{ Preliminary on the $\alpha$-order fractal differential equation \label{2g}}

A relation of the form:
\begin{equation}\label{retytte}
  G(x,f(x),D_{F,x}^{\alpha}f, D_{F,x}^{2\alpha}f(x),...D_{F,x}^{n\alpha}f(x))=0,
\end{equation}
where $f(x)$ is the unknown function and $G$ is an assigned function of the $n+2$ variables: $x, f(x), D^{\alpha}_{F,x} f(x), D^{2\alpha}_{F,x} f(x), \cdots, D^{n\alpha}_{F,x}f(x)$,  is called  $n^{th}\,\alpha$-order  ordinary fractal differential equation; here $n^{th}$ indicate the maximum order of the $F^{\alpha}$-derivative  that appears in the equation while,  as usual,  $\alpha$ denotes the dimension of the fractal subset of the real line, $0<\alpha\leq 1.$
  If $G$ is a first-degree polynomial of the variables $f(x),D_{F,x}^{\alpha}f(x),D_{F,x}^{2\alpha}f(x),...D_{F,x}^{n\alpha}f(x),$ the equation is said to be linear and its general form is:
\begin{equation}\label{iortppk}
  a_{0}(x)D_{F,x}^{n\alpha}f(x)+ a_{1}(x)D_{F,x}^{n(\alpha-1)}f(x)+...+
  a_{n}(x)f(x)=g(x),
\end{equation}
where the coefficients $a_i(x),$ for $i=0,1,\cdots,n$ are $F$-continuous functions in $(a,b)$ as well as the function $g(x).$ Whenever $g(x)=0,$ the equation \eqref{iortppk} is said to be homogeneous.
An equation that is not of the form of the equation \eqref{iortppk}  is said a nonlinear  $n^{th}\,\alpha$-order fractal differential equation.

\begin{example}
The following third $\alpha$-order  fractal differential equation
\begin{equation}\label{iop}
  D_{F,x}^{3\alpha}f(x)+2\exp(S_{F}^{\alpha}(x))D_{F,x}^{2\alpha}f(x)+3
  D_{F,x}^{\alpha}f(x)=S_{F}^{\alpha}(x)^{4}
\end{equation}
is linear.
\end{example}

\begin{example}
The following second $\alpha$-order  fractal differential equation represents the motion of an oscillating pendulum with fractal time
\begin{equation}\label{re4ee}
  D_{F, t}^{2\alpha}\theta(t)+\frac{g}{L}\sin(\theta(t))=0,
\end{equation}
where $\theta(t)$ is the unknown function that physically means an oscillating pendulum of length $L$  with the vertical direction, is nonlinear because of the term $\sin(\theta(t))$.
\end{example}

\begin{remark}

Now let us observe that,  if in the equation \eqref{retytte} it is possible to explicit the $n^{th}\, F^{\alpha}$ derivative, therefore the $n^{th}\ \alpha$-order fractal differential equation is said to be written in normal form and we have:

\begin{equation}\label{rtteiopl}
  D_{F,x}^{n\alpha} f(x)=g(x,f(x),D_{F,x}^{\alpha} f(x),...,D_{F,x}^{(n-1)\alpha} f(x)),~~~\forall x\in F.
\end{equation}

where $g$ is an assigned function of the $n+1$ variables:

$$x, f(x), D^{\alpha}_{F,x} f(x), D^{2\alpha}_{F,x} f(x), \cdots, D^{(n-1)\alpha}_{F,x}f(x)$$

\end{remark}

\begin{definition}

A solution of the $n^{th}\ \alpha$-order fractal differential equation is a function $\psi:(a,b)\rightarrow \mathbb{R}$ such that $\psi\in C^{n,\alpha}(a,b)$ and such that:

$G(x, \psi (x), D_{F,x}^{\alpha}\psi,(x) (D_{F,x}^{2\alpha})\psi,(x) \cdots (D_{F,x}^{n\alpha})\psi(x))=0, \ \ \ \  \forall x\in F$.
\end{definition}

\begin{remark}
It is trivial to observe that if \eqref{iortppk} is written in a normal form, i.e. in the form of the equation \eqref{rtteiopl}, a solution of \eqref{rtteiopl}  is a function $\psi:(a,b) \rightarrow \mathbb{R}$ such that $\psi\in C^{n,\alpha}(a,b)$ and such that:
\begin{equation}
  D_{F,x}^{n\alpha} \psi(x)=g(x,\psi(x),D_{F,x}^{\alpha} \psi(x),...,D_{F,x}^{(n-1)\alpha} \psi(x)),~~~\forall x\in F.
\end{equation}

\end{remark}

We can conjecture that the set of all solutions of a $n^{th}\ \alpha$-order fractal differential equation is represented by a family of functions depending on $n$ parameters.

For simplicity,  from now on we denote by $D_{x}^{\alpha} f(x)$,  the $F^{\alpha}$-derivative of $f$ at the point $x\in F$  and we develop our theory only in the linear case and only in the case $n=2.$

\subsection{Second $\alpha$-order fractal differential equation with constant coefficient}

Let
\begin{equation}\label{iopnytcr}
aD_{x}^{2\alpha}f(x)+bD_{x}^{\alpha}f(x)+cf(x)=0.
\end{equation}
a homogeneous second $\alpha$-order fractal differential equation with constant coefficients.
We are interested in the problem of finding all solutions of Eq.\eqref{iopnytcr}.  To do that let us note that
the linearity of the differential equation means that if $\psi_1(x)$ and $\psi_2(x)$ are any two solutions of Eq.\eqref{iopnytcr} and $c_1$ and $c_2$ are any two constants, then the function $\psi(x)=c_1\psi_1(x)\,+\,c_2\psi_2(x)$ is also a solution of Eq. \eqref{iopnytcr}. Moreover, let us recall that if $\frac{\psi_1(x)}{\psi_2(x)}\neq k,$ where $k$ is a generic constant,  $\forall x\in (a,b)$ therefore the functions $\psi_1(x)$ and $\psi_2(x)$ are linearly independent on $(a,b).$ By a similarity with the first $\alpha$-order fractal differential equation (see \cite{golmankhaneh2023solving}),  we assume that the function $\psi_1(x)=\exp(rS^{\alpha}_F(x))$ is a solution of Eq.\eqref{iopnytcr}. Therefore,  by substituting $\psi_1(x)$ into Eq.\eqref{iopnytcr} we have:

\begin{equation}\label{1iop}
(ar^{2}+br+c)\exp(rS_{F}^{\alpha}(x))=0.
\end{equation}
Since $\exp(rS_{F}^{\alpha}(x))$ is never zero, we can conclude:
\begin{equation}\label{iopm}
(ar^{2}+br+c)=0.
\end{equation}
This equation is known as the characteristic equation for Eq.\eqref{iopnytcr}. We encounter three main cases:

1. \textbf{Real Roots Case} ($b^2-4ac>0$): In this scenario, let's assume that $r_{1}$ and $r_{2}$ are the two real distinct root solutions of Eq.\eqref{iopm}.  Therefore the functions $\psi_1(x)=\exp (r_{1}S_{F}^{\alpha}(x))$ and $\psi_2(x)=\exp (r_{2}S_{F}^{\alpha}(x))$ are two linear indipendent solutions of Eq. \eqref{iopnytcr}.
By the conjecture we did we get that the general solution of Eq. \eqref{iopnytcr} is given by:

\begin{equation}\label{zaqqq}
\psi (x)=c_{1}\psi_{1}(x)+c_{2}\psi_{2}(x)=c_{1}\exp(r_{1}S_{F}^{\alpha}(x))
+c_{2}\exp(r_{1}S_{F}^{\alpha}(x)),
\end{equation}

2. \textbf{Complex Roots Case} ($b^2-4ac<0$): In this case,  let us assume that $r_{1}=\lambda+i\nu$ and $r_{2}=\lambda-i\nu$, where $\lambda$ and $\nu$ are real numbers, are the two real distinct root solutions of Eq.\eqref{iopm}. Therefore  the functions $$\psi_1(x)=\exp ((\lambda+i\nu)S_{F}^{\alpha}(x))=\exp (\lambda S^{\alpha}_F(x))\left (\cos\nu S^{\alpha}_F(x)+i\ \sin\nu S^{\alpha}_F(x)\right) $$
and
$$\psi_2(x)=\exp ((\lambda-i\nu)S_{F}^{\alpha}(x))=\exp (\lambda S^{\alpha}_F(x))\left (\cos\nu S^{\alpha}_F(x)-i\ \sin\nu S^{\alpha}_F(x)\right )$$ are two linear independent solutions of Eq. \eqref{iopnytcr}.
 Since we are considering only real functions, by the elementary property of the complex numbers and by the conjecture we did about the set of all solutions of the second $\alpha$- order fractal differential equation  we get that the functions:

\begin{equation}\label{er5ds}
\frac{1}{2}(\psi_{1}(x)+\psi_{2}(x))=\exp (\lambda S^{\alpha}_F(x)\cos(\nu  S^{\alpha}_F(x))
\end{equation}

and

\begin{equation}\label{erds}
\frac{1}{2i}(\psi_{1}(x)+\psi_{2}(x))=\exp (\lambda S^{\alpha}_F(x))\sin(\nu S^{\alpha}_F(x))
\end{equation}

are still solutions of Eq.\eqref{iopnytcr}. Therefore,
 we can represent the general solution of Eq.\eqref{iopnytcr}, case: $b^2-4ac<0$,  in terms of trigonometric functions:

\begin{equation}\label{reee}
\psi (x)=c_{1}\exp(\lambda S_{F}^{\alpha}(x))\cos(\nu S_{F}^{\alpha}(x))+
c_{2}\exp(\lambda S_{F}^{\alpha}(x))\sin(\nu S_{F}^{\alpha}(x))
\end{equation}

3. \textbf{Repeated Roots Case} ($b^2-4ac=0$): In this situation, we have $r_{1}=r_{2}=-b/2a$.  Therefore,  the general solution for Eq.\eqref{iopnytcr} takes the form:

\begin{equation}\label{rtt}
\psi(x)=c_{1}\exp\left(-\frac{b S_{F}^{\alpha}(x)}{2a}\right)+
c_{2}S_{F}^{\alpha}(x)\exp\left(-\frac{b S_{F}^{\alpha}(x)}{2a}\right)
\end{equation}

Indeed, let us observe that if $v(x)$ is a real function on $(a,b)$ that is also two-times $F^{\alpha}$-differentiable at $x\in F$, indicated by $\psi_2(x)=v(x)\exp(-b S_{F}^{\alpha}(x)/2a)$,  it follows that the functions
 $\psi_1(x)=\exp(-b S_{F}^{\alpha}(x)/2a)$ and the function $\psi_2(x)$ are linearly independent, since $\frac{\psi_2(x)}{\psi_1(x)}\,=\,v(x).$ So replacing $\psi_2(x)$ in the Eq.\eqref{iopnytcr} we have:

\begin{equation}\label{retaq}
D_{x}^{2\alpha}v(x)=0,
\end{equation}

and solving this equation for $v(x)$ we yields:

\begin{equation}\label{e55rsw}
v(x)=c_{1}+c_{2}S_{F}^{\alpha}(x)
\end{equation}

This completes the solutions for the second $\alpha$-order fractal differential equation with constant coefficients under various cases.

\begin{example}
Consider the second $\alpha$-order fractal differential equation:
\begin{equation}\label{oplmhn}
  D_{x}^{2\alpha}f(x)+D_{x}^{\alpha}f(x)+f(x)=0
\end{equation}
Its characteristic equation is given by:
\begin{equation}\label{io4758}
   r^2+r+1=0
\end{equation}
The roots of this characteristic equation are:
\begin{equation}\label{iopllk}
   r_1=-\frac{1}{2}+ i\frac{\sqrt{3}}{2}~  \textmd{and} ~ r_2=-\frac{1}{2}- i\frac{\sqrt{3}}{2}
\end{equation}
Therefore, the general solution of Eq.\eqref{oplmhn} is:
\begin{align}\label{iomnyhg}
   \psi(x)&=c_{1}\exp\left(-\frac{S_{F}^{\alpha}(x)}{2}\right)
   \cos\left(\frac{\sqrt{3}S_{F}^{\alpha}(x)}{2}\right)+
   c_{2}\exp\left(-\frac{S_{F}^{\alpha}(x)}{2}\right)
   \sin\left(\frac{\sqrt{3}S_{F}^{\alpha}(x)}{2}\right)\nonumber\\&
   \propto c_{1}\exp\left(-\frac{x^{\alpha}}{2}\right)
   \cos\left(\frac{\sqrt{3}x^{\alpha}}{2}\right)+
   c_{2}\exp\left(-\frac{x^{\alpha}}{2}\right)
   \sin\left(\frac{\sqrt{3}x^{\alpha}}{2}\right)
\end{align}
\begin{figure}[H]
  \centering
  \includegraphics[scale=0.5]{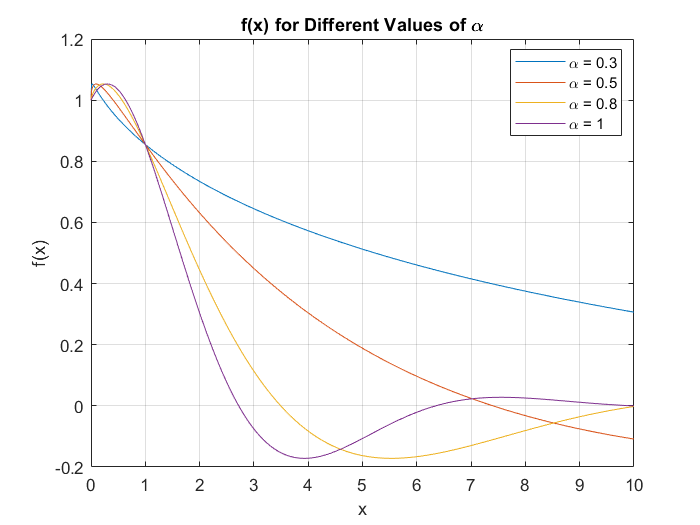}
  \caption{Graph of Eq.\eqref{iomnyhg} for $c_{1}=1,~c_{2}=1$.}\label{aqqaa}
\end{figure}
In Figure \ref{aqqaa}, we have illustrated Eq.\eqref{iomnyhg} for various values of $\alpha$, highlighting how the fractal support influences the solutions.

\end{example}

\section{Uniqueness theorem for second $\alpha$-order fractal linear differential equations \label{3g}}
In this section, we will prove the existence and the uniqueness of the solution of a given second $\alpha$-order fractal differential equation with some fixed initial conditions
\begin{theorem}
Consider the following second $\alpha$-order fractal differential equation:
\begin{equation}\label{dds}
D_{x}^{2\alpha}f(x)+a_{1}D_{x}^{\alpha}f(x)+a_{2}f(x)=0
\end{equation}
with the initial conditions
\begin{equation}\label{ijnplm}
f(x_{0})=f_{0},~~~D_{x}^{\alpha}f(x)|_{x=x_{0}}=D_{x}^{\alpha}f_{0}.
\end{equation}
Where,  $x_{0}\in (a,b)\cap F$,  and $f_{0}$, $D_{x}^{\alpha}f_{0}$ are two predefined constants.\\
Let $\psi_{1}(x)$ and $\psi_{2}(x)$ be the two linear independent solutions of Eq.\eqref{dds} (as already founded in the previous section),
therefore there exist two constants $c_1$ and $c_2$ such that $\psi(x)=c_{1}\psi_{1}(x)+c_{2}\psi_{2}(x)$ is a solution of the assigned second $\alpha$-order fractal differential equation given by Eq.\eqref{dds} with the initial conditions given by Eq.\eqref{ijnplm}
\end{theorem}

\begin{proof}
So that $\psi(x)=c_{1}\psi_{1}(x)+c_{2}\psi_{2}(x)$ is a solution of the  Eq.\eqref{dds} with the initial conditions given by  Eq.\eqref{ijnplm},
the constants $c_1$ and $c_2$ should be the solutions of the following linear system:
\begin{align}\label{tijug}
c_{1}\psi_{1}(x_{0})+c_{2}\psi_{2}(x_{0})&=f_{0}\nonumber\\
c_{1}D_{x}^{\alpha}\psi_{1}(x_{0})+c_{2}
D_{x}^{\alpha}\psi_{2}(x_{0})&=D_{x}^{\alpha}f_{0}
\end{align}
It is well known that the system \eqref{tijug} admits one and only one solution if the determinant, nominated Wronskian
\begin{align}\label{olpmnbvcxz}
W[\psi_{1},\psi_{2}]=\begin{vmatrix}
\psi_{1}(x_{0}) & \psi_{2}(x_{0}) \\
D_{x}^{\alpha}\psi_{1}(x_{0})& D_{x}^{\alpha}\psi_{2}(x_{0})
\end{vmatrix}
\end{align}
is not zero. Since the functions $\psi_{1}(x)$ and $\psi_{2}(x)$ are linearly independent,  the Wronskian condition (determinant of Wronskian different from zero), is satisfied in either case (see \cite{peano1889determinant,Apostol1974}). Therefore,  $c_{1},~c_{2}$ are the unique constants satisfying Eq.\eqref{dds} with  the initial conditions given by Eq.\eqref{ijnplm}. This shows that there is a unique linear combination of $\psi_{1}(x)$ and $\psi_{2}(x)$ which is a solution of the assigned second $\alpha$-order fractal differential equation.
\end{proof}
Now, in order to prove the uniqueness theorem we need to prove the following technical Lemma
\begin{lemma}
Let $\psi(x)$ be any solution on  $(a,b)\cap F$ of the following fractal differential equation:
\begin{equation}\label{yuttwse}
D_{x}^{2\alpha}f+a_{1}D_{x}^{\alpha}f+a_{2}f=0,
\end{equation}

 Let $x_{0}\in (a,b)\cap F.$ Then for all $x\in (a,b)\cap F$ we have
\begin{align}\label{ikmnxzsd}
    \|\psi(x_{0})\|\exp(-k|S_{F}^{\alpha}(x)&-S_{F}^{\alpha}(x_{0})|)
    \leq \|\psi(x)\|\leq \nonumber\\& \|\psi(x_{0})\|\exp(k|S_{F}^{\alpha}(x)-S_{F}^{\alpha}(x_{0})|)
  \end{align}
where
\begin{equation}\label{iuuzas}
    \|\psi(x)\|=[|\psi(x)|^2+|D_{x}^{\alpha}\psi(x)|^2]^{1/2},~~~
    k=1+|a_{1}|+|a_{2}|.
  \end{equation}
\end{lemma}

\begin{proof}
  Let $v(x)=\|\psi(x)\|^2$. This implies that
  \begin{align}\label{ijjjjn}
    v(x)&=\left(\psi(x)\right)^2+\left (D_{x}^{\alpha}\psi(x)\right )^{2}
    \end{align}

  Subsequently,
  \begin{align}\label{i95jjjjn}
 D_x^{\alpha}v(x)&=2\psi(x)D_{x}^{\alpha}\psi(x)+2D_{x}^{\alpha}\psi(x)D_{x}^{2\alpha}\psi(x)
  \end{align}
  and

  \begin{equation}\label{olmnvcx}
    \left | D_{x}^{\alpha}v(x)     \right |\leq 2|\psi(x)|\left| D_{x}^{\alpha}\psi(x)  \right |+2\left |  D_{x}^{\alpha}\psi(x)  \right |\ \left | D_{x}^{2\alpha}\psi(x)     \right |.
  \end{equation}
  By the hypothesis on $\psi(x)$ we have:
  \begin{equation}\label{olkmjuyh}
    D_{x}^{2\alpha}\psi(x)=-a_{1}\, D_{x}^{\alpha}\psi(x)\ -\ a_{2}\psi(x),
  \end{equation}
  and, consequently,
  \begin{equation}\label{olmnvzsw}
    \left | D_{x}^{2\alpha}\psi(x) \right|\leq |a_{1}|\left|D_{x}^{\alpha}\psi(x)\right|\ +\ |a_{2}||\psi(x)|.
  \end{equation}
  Now, substring  \eqref{olmnvzsw} into \eqref{olmnvcx} we obtain
  \begin{equation}\label{ilok}
    \left |D_{x}^{\alpha}v(x)\right |\leq 2(1+|a_{2}|)|\psi(x)|\left |D_{x}^{\alpha}\psi(x)\right |+2|a_{1}|\left |D_{x}^{\alpha}\psi(x)   \right |^2.
  \end{equation}
  and by the inequality
  \begin{equation}\label{xxxxz}
    2|b||c|\leq |b|^2+|c|^2,
  \end{equation}
  setting $b=\psi(x)$ and $c=\, D_{x}^{\alpha}\psi(x)$, we get
  \begin{align}\label{lllll}
    \left |D_{x}^{\alpha}v(x)\right |&\leq (1+|a_{2}|)|\psi(x)|^{2}+(1+2|a_{1}|+|a_{2}|)\left | D_{x}^{\alpha}\psi(x)\right |\nonumber\\
    &\leq 2(1+|a_{1}|+|a_{2}|)[|\psi(x)|^2+|\psi^{\alpha}(x)|^2],
  \end{align}
  therefore
  \begin{equation}\label{ookkmn}
    \left | D_{x}^{\alpha}v(x)  \right |\leq 2kv(x).
  \end{equation}
 so we get
  \begin{equation}\label{ppppllk}
    -2kv(x)\leq v^{\alpha}(x)\leq 2kv(x)
  \end{equation}
  Now to prove  Eq.\eqref{ikmnxzsd}, let us observe that the right inequality of Eq. \eqref{ppppllk} can be expressed as
  \begin{equation}\label{i9}
    D_{x}^{\alpha}v(x)-2k v(x)\leq 0.
  \end{equation}
Since $\exp(-2kS^{\alpha}_F(x))\geq 0$ for all $x\in (a,b)\cap F$ we can multiply both members of Eq \eqref{i9} by $\exp(-2kS^{\alpha}_F(x)$ so we get that
  \begin{equation}
    D^{\alpha}_x\ v(x)\exp(-2kS_{F}^{\alpha})(x)\leq 0.
  \end{equation}
  Now, let $x\in (a,b)\cap F.$  Let us start by considering $x>x_{0},$ therefore the $F^{alpha}$- integral  from $x_{0}$ to $x$ of  $D^{\alpha}_x v(x)\exp(-2kS_{F}^{\alpha})(x)$   results
 \begin{equation}
    \exp(-2kS_{F}^{\alpha}(x))v(x)-\exp(-2kS_{F}^{\alpha}(x_{0}))v(x_{0})\leq 0.
  \end{equation}
  This leads to the inequality
  \begin{equation}\label{ooop}
   v(x)\leq v(x_{0})\exp(2k(S_{F}^{\alpha}(x)-S_{F}^{\alpha}(x_{0})))
  \end{equation}
  which yielding
  \begin{equation}\label{87ooop}
   \|\psi(x)\|\leq \|\psi(x_{0})\|\exp(2k(S_{F}^{\alpha}(x)-S_{F}^{\alpha}(x_{0}))),~~~(x>x_{0}).
  \end{equation}
  The corresponding left inequality in Eq.\eqref{ppppllk} implies
  \begin{equation}\label{8mn7ooop}
   \|\psi(x_{0})\|\exp(-2k(S_{F}^{\alpha}(x)-S_{F}^{\alpha}(x_{0})))\leq \|\psi(x)\|,~~~(x>x_{0}),
  \end{equation}
  which consequently leads to
  \begin{align}
   \|\psi(x_{0})\|\exp(-2k(S_{F}^{\alpha}(x)&-S_{F}^{\alpha}(x_{0})))\leq \|\psi(x)\|\leq \nonumber\\& \|\psi(x_{0})\|\exp(2k(S_{F}^{\alpha}(x)-S_{F}^{\alpha}(x_{0}))),
  \end{align}
  Now, considering Eq.\eqref{ppppllk} for $x<x_{0}$ along with a fractal integration from $x$ to $x_{0}$, we obtain
  \begin{align}\label{uuuuuuu}
   \|\psi(x_{0})\|\exp(2k(S_{F}^{\alpha}(x)&-S_{F}^{\alpha}(x_{0})))\leq \|\psi(x)\|\leq \nonumber\\& \|\psi(x_{0})\|\exp(-2k(S_{F}^{\alpha}(x)-S_{F}^{\alpha}(x_{0}))),
  \end{align}
  which completes the proof.
\begin{figure}[H]
  \centering
  \includegraphics[scale=0.5]{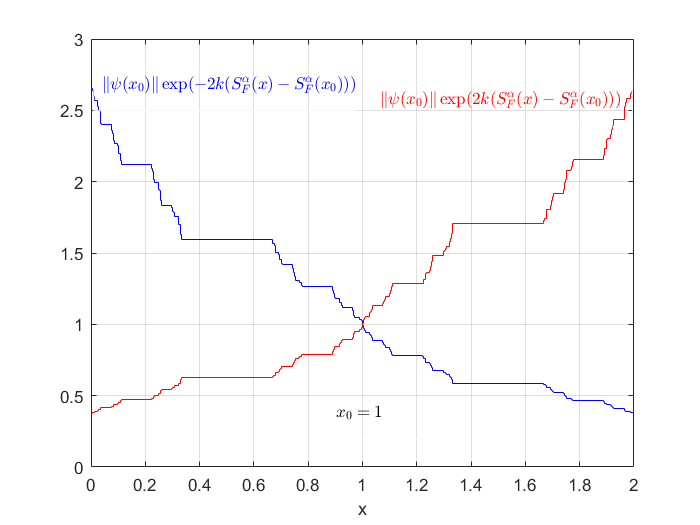}
  \caption{Graph of Eq.\eqref{ikmnxzsd}}\label{ffrrgf}
\end{figure}

In Figure \ref{ffrrgf}, it is evident that $|\psi(x)|$ consistently remains between the two curves $|\psi(x_{0})|\exp(2k(S_{F}^{\alpha}(x)-S_{F}^{\alpha}(x_{0})))$ and $|\psi(x_{0})|\exp(-2k(S_{F}^{\alpha}(x)-S_{F}^{\alpha}(x_{0})))$.

\end{proof}

\begin{theorem}
  (Uniqueness Theorem)
   Let $f_{0}$ and $D_{x}^{\alpha}f_{0}$ any two constants and let $x_{0}$ be any point in the fractal set $F$. On any interval $I\subset (a,b)\cap F$ containing $x_0$ there exists at most one solution $\psi (x)$
of the initial value problem
\begin{equation}\label{reeee}
aD^{2\alpha}_x f(x)+bD^{\alpha}_xf(x)+cf(x)=0 \quad f(x_{0})=f_{0}, \quad f^{\alpha}(x_{0})=D_{x}^{\alpha}f_{0}.
\end{equation}
  \end{theorem}

  \begin{proof}
  Let us assume that there are two different solutions indicated by $\psi$ and $\phi$ of the assigned initial value problem.  Let $\chi=\psi-\phi$. Then, we have $\chi (x)$ satisfies the initial value problem with  the initial conditions $f_0 =D^{\alpha}_x f_0 = 0$.
Thus $\|\chi(x_0)\| = 0$ and by applying the inequalities Eq.\eqref{uuuuuuu}, of the previous Lemma, to $\chi (x)$, we get $\|\chi(x)\| = 0$ for all $x \in I\subset F \cap (a, b)$.
Therefore  $\psi (x) - \phi (x)=0$ for all $x\in I.$Thus, there exists at most one solution to the initial value problem \eqref{reeee}, establishing the uniqueness of the solution.
\end{proof}

\begin{example}
Consider second $\alpha$-order fractal differential equation as
\begin{equation}\label{e}
  D_{x}^{2\alpha}f(x)+5D_{x}^{\alpha}f(x)+6f(x)=0,~~~x\in F
\end{equation}
with initial condition
\begin{equation}\label{eqaioerllk}
  f(0)=2,~~~D_{x}^{\alpha}f_{0}=3,
\end{equation}
By using Eqs.\eqref{zaqqq} and  \eqref{eqaioerllk} we obtain solution of Eq.\eqref{e} as follows
\begin{align}\label{eq12}
  \psi(x)&=9\exp(-2S_{F}^{\alpha}(x))-7\exp(-3S_{F}^{\alpha}(x))\nonumber\\
  &\propto 9\exp(-2x^{\alpha})-7\exp(-3x^{\alpha})
\end{align}
 The figures illustrating the exact and approximate solutions of Eq. \eqref{e} can be observed in Figure \ref{95147} and Figure \ref{95147jj}, respectively.
\begin{figure}[H]
  \centering
 \includegraphics[scale=0.5]{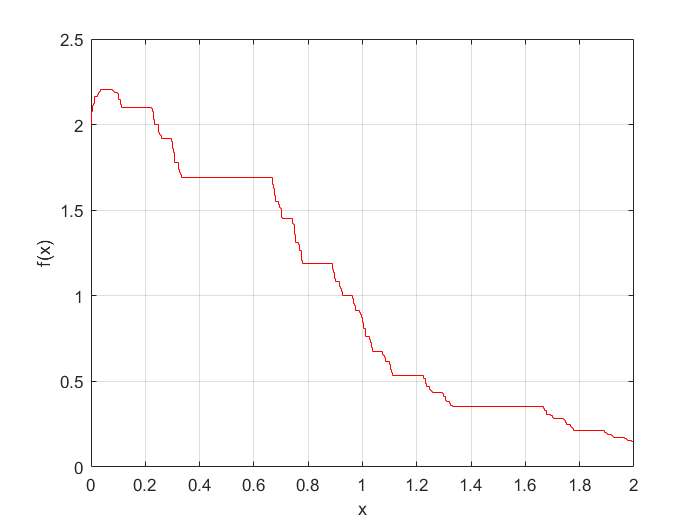}
  \caption{Solution of Eq.\eqref{e} with initial condition Eq.\eqref{eqaioerllk} for $\alpha=0.63$.
    }\label{95147}
\end{figure}
\begin{figure}[H]
  \centering
  \includegraphics[scale=0.5]{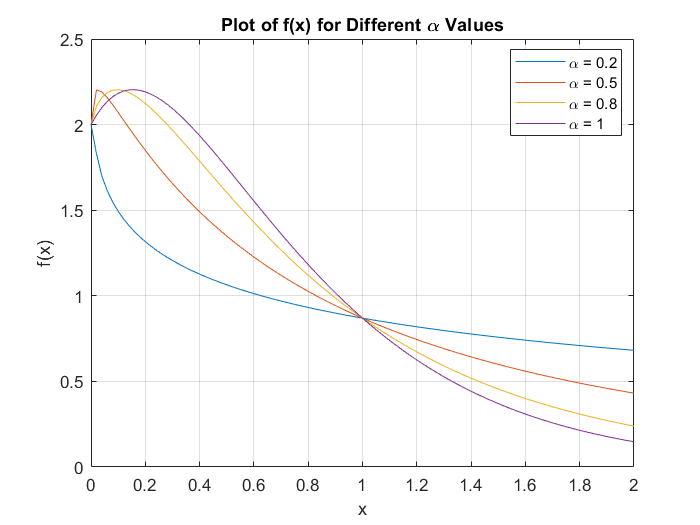}
  \caption{Approximate solution of Eq.\eqref{e} with initial condition Eq.\eqref{eqaioerllk} for different value of $\alpha$.
    }\label{95147jj}
\end{figure}
\end{example}
\begin{definition}
We say that the fractal dimension of the solution space of the second $\alpha$-order fractal differential equation is  $2\alpha$.
\end{definition}

\section{Exact Second $\alpha$-Order Fractal Differential Equation \label{4g}}
In this section, we establish the concept of homogeneous fractal calculus and describe the exact second $\alpha$-order fractal differential equation with $\alpha$ as a key parameter.

\begin{definition}
A  second $\alpha$-order fractal differential equation takes the form:
\begin{equation}\label{rreeetyi}
P(x)D_{x}^{2\alpha}f(x)+Q(x)D_{x}^{\alpha}f(x)+R(x)f(x)=0,
\end{equation}
It is considered exact when it can be transformed into the following form:
\begin{equation}\label{reeeaqw}
D_{x}^{\alpha}(P(x)D_{x}^{\alpha}f(x))+D_{x}^{\alpha}(g(x)f(x))=0,
\end{equation}
Where the function $g(x)$ can be determined in terms of $P(x)$, $Q(x)$, and $R(x)$.
\end{definition}

\begin{theorem}
The second $\alpha$-order fractal  differential equation, as given by Equation \eqref{rreeetyi}:
\begin{equation}\label{rtyi}
P(x)D_{x}^{2\alpha}f(x)+Q(x)D_{x}^{\alpha}f(x)+R(x)f(x)=0,
\end{equation}
is exact if
\begin{equation}\label{rqaweeelmnb}
D_{x}^{2\alpha}P(x)-D_{x}^{\alpha}Q(x)+R(x)=0.
\end{equation}
In other words, the equation is exact when the combination of the second $F^{\alpha}$-derivative of the coefficient function $P(x)$ and the $F^{\alpha}$-derivative of the coefficient function $Q(x)$, along with the term $R(x)$, equals zero.
\end{theorem}
\begin{proof}
  To prove this theorem, we can equate the coefficients of Equations \eqref{rtyi} and \eqref{reeeaqw} and then eliminate $g(x)$, resulting in the equation:
  \begin{equation}
  D_{x}^{2\alpha}P(x)-D_{x}^{\alpha}Q(x)+R(x)=0.
  \end{equation}
\end{proof}

\begin{theorem}
Consider a second $\alpha$-order fractal homogeneous equation given by:
\begin{equation}\label{t44tyy}
  P(x)D_{x}^{2\alpha}f(x)+Q(x)D_{x}^{\alpha}f(x)+R(x)=0,
\end{equation}
If this equation is not exact, we can make it exact by multiplying it by a function $\mu(x)$, which is a solution of the following equation, often referred to as the adjoint equation associated with Equation \eqref{t44tyy}:
\begin{align}\label{1ioplki369}
 & P(x)D_{x}^{2\alpha}\mu(x)+(2 D_{x}^{\alpha}P(x)-Q(x))D_{x}^{\alpha}\mu(x)\nonumber\\
 &\quad+(D_{x}^{2\alpha}P(x)-D_{x}^{\alpha}Q(x)+R(x))\mu(x)=0.
\end{align}
\end{theorem}
\begin{proof}
  Consider the equation obtained by multiplying the given second $\alpha$-order fractal linear homogeneous equation by $\mu(x)$:

\begin{equation}\label{io33pplk}
\mu(x)P(x)D_{x}^{2\alpha}f(x)+\mu(x)Q(x)D_{x}^{\alpha}f(x)+\mu(x)R(x)=0,
\end{equation}

We can express Equation \eqref{io33pplk} in the following form:

\begin{equation}\label{rt655}
D_{x}^{\alpha}(\mu(x)P(x)D_{x}^{\alpha}f(x))+D_{x}^{\alpha}(g(x)f(x))=0,
\end{equation}

By equating the coefficients of Equations \eqref{io33pplk} and \eqref{rt655}, we can eliminate the function $g(x)$, revealing that the function $\mu(x)$ must satisfy the following equation:

\begin{align}\label{ioplki369}
& P(x)D_{x}^{2\alpha}\mu(x)+(2 D_{x}^{\alpha}P(x)-Q(x))D_{x}^{\alpha}\mu(x)\nonumber\\
&\quad+(D_{x}^{2\alpha}P(x)-D_{x}^{\alpha}Q(x)+R(x))\mu(x)=0.
\end{align}

This completes the proof, demonstrating the relationship between the function $\mu(x)$ and the original equation.

\end{proof}

\begin{lemma}
A second $\alpha$-order fractal  homogeneous equation, represented as:

\begin{equation}\label{ttrre23yy}
P(x)D_{x}^{2\alpha}f(x)+Q(x)D_{x}^{\alpha}f(x)+R(x)=0,
\end{equation}

can be referred to as self-adjoint if it satisfies the condition:

\begin{equation}\label{eww}
D_{x}^{\alpha}P(x)=Q(x).
\end{equation}
In simpler terms, the equation is considered self-adjoint when the $\alpha$-derivative of the coefficient function $P(x)$ is equal to the function $Q(x)$.
\end{lemma}
\begin{proof}
 The proof follows straightforwardly from Equation \eqref{ioplki369}.
\end{proof}
\begin{example}
Consider a second $\alpha$-order fractal differential equation given by:

\begin{equation}\label{iopl8597}
16 D_{x}^{2\alpha}f(x)-8 D_{x}^{\alpha}f(x)+145 f(x)=0,
\end{equation}

with the following initial conditions:

\begin{equation}\label{iiiii951}
f(0)=-2,~~~D_{x}^{\alpha}f(x)|_{x=0}=1.
\end{equation}

The solution to this equation is given by:

\begin{align}\label{oollop}
f(x)&=-2\exp\bigg(\frac{S_{F}^{\alpha}(x)}{4}\bigg)\cos(3S_{F}^{\alpha}(x))
+\frac{1}{2}\exp\bigg(\frac{S_{F}^{\alpha}(x)}{4}\bigg)\sin(3S_{F}^{\alpha}(x))
\nonumber\\& \propto -2\exp\bigg(\frac{x^{\alpha}}{4}\bigg)\cos(3x^{\alpha}))
+\frac{1}{2}\exp\bigg(\frac{x^{\alpha}}{4}\bigg)\sin(3x^{\alpha})
\end{align}

\begin{figure}[H]
  \centering
  \includegraphics[scale=0.5]{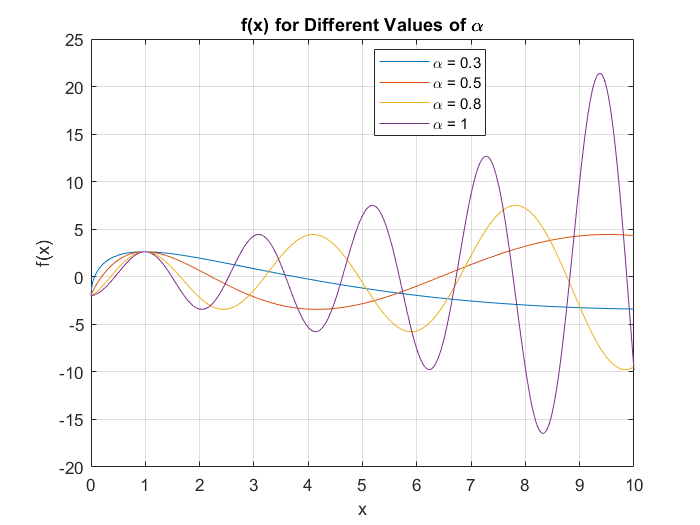}
  \caption{Graph of Eq.\eqref{oollop}.}\label{dsswwsa}
\end{figure}
In Figure \ref{dsswwsa}, we have graphed Eq.\eqref{oollop}, illustrating how varying dimensions influence the solution.
In simpler terms, this example illustrates a second $\alpha$-order fractal differential equation with specific initial conditions and provides the corresponding solution expressed in terms of trigonometric and exponential functions.
\end{example}

\begin{theorem}
  Let's consider a solution, denoted as $f_{1}(x)$, to a second $\alpha$-order fractal differential equation given by:
\begin{equation}\label{23eeeee}
D_{x}^{2\alpha}f(x)+p(x)D_{x}^{\alpha}f(x)+q(x)f(x)=0.
\end{equation}
Now, the second solution, denoted as $f(x)$, can be expressed as:
\begin{equation}\label{xza2}
f(x)=v(x)f_{1}(x),
\end{equation}
which satisfies the following equation:
\begin{equation}\label{erxsazq}
f_{1}(x)D_{x}^{2\alpha}v(x)+(2D_{x}^{\alpha} f_{1}(x)+p(x)f_{1}(x))D_{x}^{\alpha}v(x)=0.
\end{equation}
In essence, this theorem explains how to find a second solution to a second $\alpha$-order fractal differential equation when you already have one solution, and it provides a differential equation for the function $v(x)$ that relates the two solutions.
\end{theorem}
\begin{proof}
  To prove this theorem,  We can find the $F^{\alpha}$-derivatives of $f(x)$ with respect to $x$ as follows:
\begin{align}\label{iiki}
D_{x}^{\alpha}f(x) &= v(x)D_{x}^{\alpha}f_{1}(x) + D_{x}^{\alpha}v(x)f_{1}(x), \nonumber\\
D_{x}^{2\alpha}f(x) &= v(x)D_{x}^{2\alpha}f_{1}(x) + 2D_{x}^{\alpha}v(x)D_{x}^{\alpha}f_{1}(x) + D_{x}^{2\alpha}v(x)f_{1}(x).
\end{align}
Now, we substitute Equations \eqref{iiki} back into Eq.\eqref{23eeeee}:
\begin{align}
&f_{1}(x)D_{x}^{2\alpha}v(x)+(2D_{x}^{\alpha}f_{1}(x)+p(x)f_{1}(x))D_{x}^{\alpha}v(x)\nonumber\\&+
(D_{x}^{2\alpha}f_{1}(x)+p(x)D_{x}^{\alpha}f_{1}(x)+qf_{1}(x))v(x)=0
\end{align}
Since $f_{1}(x)$ is solution of Eq.\eqref{23eeeee}, then we have
\begin{equation}\label{2erxsazq}
f_{1}(x)D_{x}^{2\alpha}v(x) + (2D_{x}^{\alpha}f_{1}(x) + p(x)f_{1}(x))D_{x}^{\alpha}v(x) = 0.
\end{equation}
 This completes the proof of the theorem.
\end{proof}
\begin{example}
Let's consider the equation:

\begin{equation}\label{r82plm}
2S_{F}^{\alpha}(x)^{2}D_{x}^{2\alpha}f(x)+
3S_{F}^{\alpha}(x)D_{x}^{\alpha}f(x)-f(x)=0,
\end{equation}

where $f_{1}(x) = S_{F}^{\alpha}(x)^{-1}$ is one of the solutions.

To find a second fundamental solution, we propose $f(x)=v(x)S_{F}^{\alpha}(x)^{-1}$. Substituting this expression for $f(x)$, $D_{x}^{\alpha}f(x)$, and $D_{x}^{2\alpha}f(x)$ into Equation \eqref{r82plm} and collecting terms, we obtain:

\begin{equation}\label{iiiiii}
2S_{F}^{\alpha}(x)D_{x}^{2\alpha}v(x)-D_{x}^{\alpha}v(x)=0.
\end{equation}

Now, if we let $w(x)=D_{x}^{\alpha}v(x)$, we have a separable $\alpha$-order differential equation. Solving it, we find:

\begin{equation}\label{iii555}
w(x)=cS_{F}^{\alpha}(x)^{1/2}.
\end{equation}

From this, we can determine $v(x)$ as:

\begin{equation}\label{iooi951}
v(x)=\frac{2}{3}cS_{F}^{\alpha}(x)^{3/2}+k,
\end{equation}

and consequently, the solution $f(x)$ becomes:

\begin{align}\label{ppppppp}
f(x)&=\frac{2}{3}cS_{F}^{\alpha}(x)^{1/2}+kS_{F}^{\alpha}(x)^{-1}\nonumber\\&
\propto \frac{2}{3}c x^{\alpha/2}+kx^{-\alpha}
\end{align}
\begin{figure}
  \centering
  \includegraphics[scale=0.5]{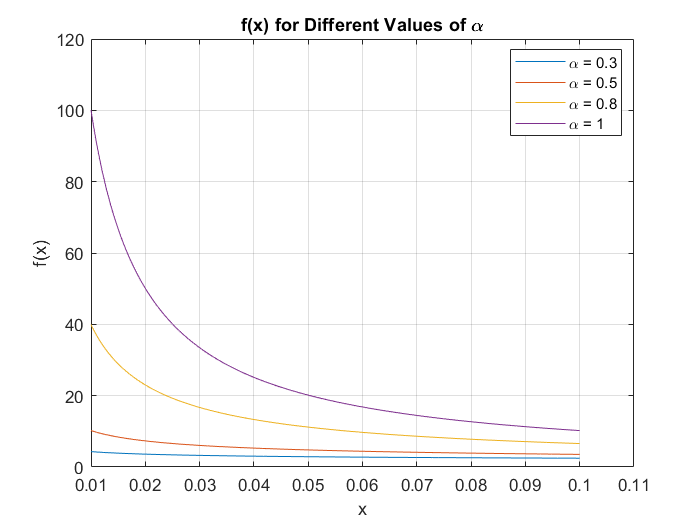}
  \caption{Graph of $f(x)\propto \frac{2}{3}c x^{\alpha/2}+kx^{-\alpha}$ for $c=k=1$. }\label{rrrcxxdde}
\end{figure}

In Figure labeled as Figure \ref{rrrcxxdde}, we have depicted an approximation of the equation referenced as Eq.\eqref{ppppppp}.
In summary, this example demonstrates how to find a second solution to the given differential equation, building on the knowledge of the first solution $f_{1}(x)$.
\end{example}
\section{Nonhomogeneous Second $\alpha$-order Fractal Differential Equation\label{5g}}
In this section, we will introduce and discuss nonhomogeneous second $\alpha$-order fractal differential equations. These equations involve both the second $\alpha$-order derivative of a function and a nonhomogeneous term, typically denoted as $g(x)$. Nonhomogeneous equations are important in modeling real-world phenomena where external influences or sources contribute to the behavior of the system. We will explore various aspects of these equations, including their solutions and properties.
\begin{definition}
Let us consider a nonhomogeneous $\alpha$-order   fractal differential equation as
\begin{equation}\label{ttyytre}
  L[f]=D_{x}^{2\alpha}f(x)+p(x)D_{x}^{\alpha}f(x)+q(x)f(x)=g(x),
\end{equation}
where $p,q$ and $g$ are given $F$-continuous on the $(a,b)$. If $g(x)=0$, namely,
\begin{equation}\label{olmnbvxz}
  L[f]=D_{x}^{2\alpha}f(x)+p(x)D_{x}^{\alpha}f(x)+q(x)f(x)=0
\end{equation}
which is called the homogenous fractal differential equation.
\end{definition}
\begin{theorem}\label{yuujjjjju84}
Consider a non-homogeneous linear fractal differential equation given by Equation \eqref{ttyytre}, where $F_{1}(x)$ and $F_{2}(x)$ are two solutions. Then, their difference $F_{1}-F_{2}$ is a solution of the corresponding homogeneous fractal differential equation, as given by Equation \eqref{olmnbvxz}.
Furthermore, if $f_{1}$ and $f_{2}$ form a fundamental set of solutions of the homogeneous fractal differential equation, then the difference $F_{1}-F_{2}$ can be expressed as:
\begin{equation}\label{ipmnbvczazaqw}
F_{1}-F_{2}=c_{1}f_{1}+c_{2}f_{2},
\end{equation}
where $c_{1}$ and $c_{2}$ are constants.
\end{theorem}
\begin{proof}
  To prove this theorem, we start by noting that for the non-homogeneous linear fractal differential equation:

\begin{equation}\label{iiiq}
L[F_{1}]=g(x),~~~ L[F_{2}]=g(x),
\end{equation}
where $L$ represents the differential operator, and $g(x)$ is the non-homogeneous term, we have two solutions $F_{1}(x)$ and $F_{2}(x)$.
Now, by subtracting these equations \eqref{iiiq}, we obtain:
\begin{equation}\label{oippllm}
L[F_{1}]-L[F_{2}]=L[F_{1}-F_{2}]=0.
\end{equation}

Then, we can conclude that:

\begin{equation}\label{ipmnb887vczazaqw}
F_{1}-F_{2}=c_{1}f_{1}+c_{2}f_{2},
\end{equation}

where $c_{1}$ and $c_{2}$ are constants. This completes the proof.
\end{proof}
\begin{theorem}
  The general solution of the nonhomogeneous fractal differential equation given by \eqref{olmnbvxz} can be represented in the form:

\begin{equation}\label{ipplllll}
f(x)=\psi(x)=c_{1}f_{1}(x)+c_{2}f_{2}(x)+F(x),
\end{equation}

Here, $f_{1}$ and $f_{2}$ are fundamental solutions of the corresponding homogeneous equation \eqref{olmnbvxz}, $c_{1}$ and $c_{2}$ are arbitrary constants, and $F(x)$ is any solution of the nonhomogeneous equation \eqref{ttyytre}. This form allows us to describe the general solution of the nonhomogeneous equation in terms of both homogeneous solutions and a particular solution.
\end{theorem}
\begin{proof}
  The proof of this theorem follows directly from Theorem \ref{yuujjjjju84}. If we consider $F_{1}$ to be $\psi(x)$ and $F_{2}$ to be $F(x)$, then we immediately obtain Equation \eqref{ipplllll}. This demonstrates that the general solution in the form stated in the theorem is indeed valid.
\end{proof}
\begin{example}
Let's consider a nonhomogeneous fractal differential equation:

\begin{equation}\label{reqazxcfr}
D_{x}^{2\alpha}f(x)-3D_{x}^{\alpha}f(x)-4f(x)=3\exp(2S_{F}^{\alpha}(x))
\end{equation}

We are looking for a particular solution $F(x)$ that satisfies the following equation:

\begin{equation}\label{ertttta}
D_{x}^{2\alpha}F(x)-3D_{x}^{\alpha}F(x)-4F(x)=3\exp(2S_{F}^{\alpha}(x))
\end{equation}

To find this particular solution $F(x)$, let us assume that the solution of Equation \eqref{ertttta} can be written as $F(x)=A\exp(2S_{F}^{\alpha}(x))$. Substituting this into Equation \eqref{ertttta}, we get:

\begin{equation}\label{rtrre58}
(4A-6A-4A)\exp(2S_{F}^{\alpha}(x)) =3\exp(2S_{F}^{\alpha}(x))
\end{equation}

Solving for $A$, we find that $A=-1/2$. Therefore, the particular solution $F(x)$ is:
\begin{equation}\label{erreeaaq}
F(x)=-\frac{1}{2}\exp(2S_{F}^{\alpha}(x))
\end{equation}
This provides a specific solution to the nonhomogeneous fractal differential equation \eqref{reqazxcfr}.
\end{example}
\begin{theorem}[Fractal Variation of Parameters]
Consider a nonhomogeneous second $\alpha$-order linear fractal differential equation:

\begin{equation}\label{oopp987}
L[f(x)]= D_{x}^{2\alpha}f(x)+p(x)D_{x}^{\alpha}f(x)+q(x)f(x)=g(x)
\end{equation}

Assuming that the functions $p(x)$, $q(x)$, and $g(x)$ are $F$-continuous on the open interval $(a,b)$, and that $f_{1}(x)$ and $f_{2}(x)$ form a fundamental set of solutions for the corresponding homogeneous fractal equation:

\begin{equation}\label{iiooppkll}
D_{x}^{2\alpha}f(x)+p(x)D_{x}^{\alpha}f(x)+q(x)f(x)=0
\end{equation}

Then, a particular solution of Equation \eqref{oopp987} can be expressed as:

\begin{equation}\label{iinnnbvcx5}
F(x)=-f_{1}\int_{x_{0}}^{x}\frac{f_{2}(s)g(s)}{W[f_{1},f_{2}]}d_{F}^{\alpha}s+
f_{2}\int_{x_{0}}^{x}\frac{f_{1}(s)g(s)}{W[f_{1},f_{2}]}d_{F}^{\alpha}s
\end{equation}

where $x_{0}$ is any conveniently chosen point in the open interval $(a,b)$. The general solution of Equation \eqref{oopp987} is then:

\begin{equation}\label{rreeeqaq}
f(x)=c_{1}f_{1}(x)+c_{2}f_{2}(x)+F(x)
\end{equation}

Here, $c_{1}$ and $c_{2}$ are constants. This theorem provides a method to find both particular and general solutions for nonhomogeneous second $\alpha$-order linear fractal differential equations under certain conditions on the functions involved.
\end{theorem}
\begin{proof}
To establish this theorem, we begin by considering a general solution for the nonhomogeneous equation:

\begin{equation}\label{traq2}
F(x)=A(x)f_{1}(x)+B(x)f_{2}
\end{equation}

Here, $A(x)$ and $B(x)$ are unknown functions, and $f_{1}$ and $f_{2}$ are the solutions to the homogeneous Eq.\eqref{iiooppkll}. Since this equation introduces two unknown functions, it's appropriate to impose an additional condition. We choose the following conditions:

\begin{equation}\label{yiiii}
D_{x}^{\alpha}A(x)f_{1}(x)+D_{x}^{\alpha}B(x)f_{2}(x)=0.
\end{equation}

Now, let us compute the fractal derivatives of $F(x)$:

\begin{align}\label{yyyyfff}
D_{x}^{\alpha} F(x)&=D_{x}^{\alpha}(A(x)f_{1}(x)+B(x)f_{2})\nonumber\\&
=A(x)D_{x}^{\alpha}f_{1}(x)+B(x)D_{x}^{\alpha}f_{2}
\end{align}

Differentiating once more:

\begin{align}\label{y33yyyfff}
D_{x}^{2\alpha} F(x)&=A(x)D_{x}^{2\alpha}f_{1}(x)+B(x)D_{x}^{2\alpha}f_{2}
\nonumber\\&+D_{x}^{\alpha}A(x)D_{x}^{\alpha}f_{1}(x)
+D_{x}^{\alpha}B(x)D_{x}^{\alpha}f_{2}
\end{align}

Now, we can express the action of $L$ on $F(x)$ as:

\begin{align}\label{yiopklji}
L[F(x)]&=A(x)L[f_{1}(x)]+B(x)L[f_{2}(x)]+
D_{x}^{\alpha}A(x)D_{x}^{\alpha}f_{1}(x)\nonumber\\&
+D_{x}^{\alpha}B(x)D_{x}^{\alpha}f_{2}
\end{align}

Since $f_{1}(x)$ and $f_{2}(x)$ are solutions to the homogeneous equation, we have:

\begin{equation}
L[F(x)]=D_{x}^{\alpha}A(x)D_{x}^{\alpha}f_{1}(x)
+D_{x}^{\alpha}B(x)D_{x}^{\alpha}f_{2}
\end{equation}

This leads to the system of equations:

\begin{equation}\label{ijbhyujn}
\left[
\begin{array}{cc}
f_{1}(x) & f_{2}(x) \\
D_{x}^{\alpha}f_{1}(x) & D_{x}^{\alpha}f_{2}(x) \\
\end{array}
\right]\left[
\begin{array}{c}
D_{x}^{\alpha}A(x) \\
D_{x}^{\alpha} B(x)\\
\end{array}
\right]=\left[
\begin{array}{c}
0 \\
g(x) \\
\end{array}
\right]
\end{equation}

To determine $A(x)$ and $B(x)$ from these conditions, we solve this system, resulting in:

\begin{align}\label{rreeww}
A(x)&=-\int_{x_{0}}^{x} \frac{f_{2}(s)g(s)}{W[f_{1},f_{2}]}d_{F}^{\alpha}s\nonumber\\
B(x)&=\int_{x_{0}}^{x} \frac{f_{2}(s)g(s)}{W[f_{1},f_{2}]}d_{F}^{\alpha}s
\end{align}
By substituting Eq.\eqref{rreeww} into Eq.\eqref{traq2}, we obtain the general solution for the nonhomogeneous equation. This concludes the proof.
\end{proof}

\begin{example}
  Consider the equation describing the motion of an undamped forced oscillator:

\begin{equation}\label{iooopkmnbvcx}
m D_{t}^{2\alpha}f(t)+k f(t)=F_{0}\cos(\omega S_{F}^{\alpha}(t))
\end{equation}

This equation represents the behavior of the oscillator, where $m, k, F_{0}$ and $\omega$ are constants. The initial conditions for this oscillator are given as:
\begin{equation}\label{iiooo74447}
f(0)=0,~~~D_{t}^{\alpha}f(t)\big|_{t=0}=0
\end{equation}
To find the general solution for Equation \eqref{iooopkmnbvcx}, which describes the oscillator's motion, we obtain the following expression:

\begin{equation}\label{ioopktfr}
f(t)=c_{1}\cos(\omega_{0} S_{F}^{\alpha}(t))+c_{2}\sin(\omega_{0} S_{F}^{\alpha}(t))+\frac{F_{0}}{m(\omega_{0}^2-\omega^2)}\cos(\omega S_{F}^{\alpha}(t))
\end{equation}

Here, $\omega_{0}=\sqrt{k/m}$ represents the natural frequency of the oscillator. By applying the initial conditions from Equation \eqref{iiooo74447}, we further simplify the solution to:

\begin{align}\label{w9586}
f(t)&=\frac{F_{0}}{m(\omega_{0}^2-\omega^2)}(\cos(\omega S_{F}^{\alpha}(t))-\cos(\omega_{0} S_{F}^{\alpha}(t)))\nonumber\\&
\propto \frac{F_{0}}{m(\omega_{0}^2-\omega^2)}(\cos(\omega t^{\alpha})-\cos(\omega_{0} t^{\alpha}))
\end{align}

In the accompanying Figure \ref{8999oplmki}, we have depicted the behavior described by Equation \eqref{w9586} for the given parameters: $\omega_{0}=1,~\omega=0.8,$ and $F_{0}=1/2$.

\begin{figure}[H]
  \centering
  \includegraphics[scale=0.5]{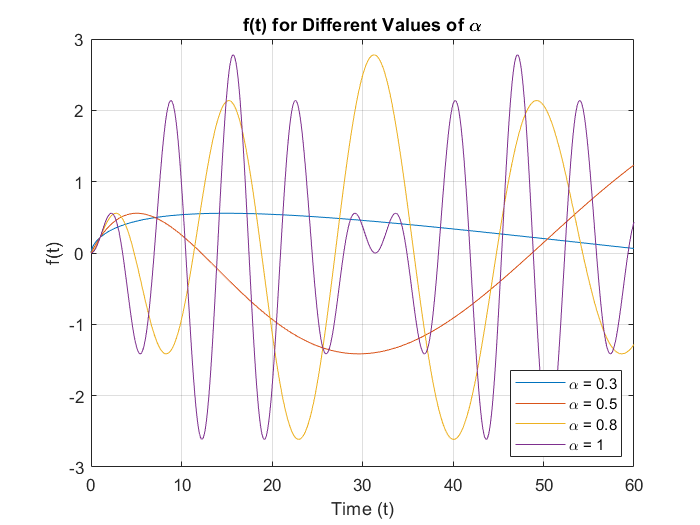}
  \caption{Graph of $f(t)=2.778\sin(0.1t^{\alpha})\sin(0.9)t^{\alpha}$}\label{8999oplmki}
\end{figure}

\end{example}

\section{Conclusion  \label{6g}}
In our paper, we have introduced the concept of second-order fractal differentials, denoted as $\alpha$-order, along with a method for solving them. We have also defined a solution space for these differentials, which encompasses non-integer dimensions. Furthermore, we have presented a uniqueness theorem for second $\alpha$-order linear fractal differential equations, and we have provided an exact formulation for a second-order fractal differential equation. This equation is complemented by its adjoint equation, making it self-adjoint. In addition, we have defined and successfully solved nonhomogeneous second $\alpha$-order fractal differential equations. The models we have presented can be applied to processes occurring in fractal time and space.\\
\textbf{Declaration of Competing Interest:}\\
The authors declare that they have no known competing financial interests or personal relationships that could have appeared to influence the work reported in this paper.\\
\textbf{CRediT author statement:}\\
Alireza.K.Golmankhnaeh : Conceptualization, Investigation, Methodology, Software, Writing- Original draft preparation.\\
Donatella Bongiorno : Investigation, Methodology, Validation, Writing- Original draft preparation, Reviewing and Editing.\\

\section{References}
\bibliographystyle{elsarticle-num}
\bibliography{Refrancesma10}






\end{document}